\documentclass{issachal}
\def\abs#1{\left \vert #1 \right \vert}
\newcommand\Prod{\mathop{\prod}\limits}
\def\pent#1#2{\lfloor{\frac{#1}{#2}}\rfloor}
\renewcommand\SS{{\bf S}} 

\newcommand\pn{\medskip\par\noindent}
\def\Frac#1#2{{\displaystyle{{#1}\over{#2}}}}
\def\[#1\]{\begin{eqnarray}#1\end{eqnarray}}
\def\Mod#1{\,(\hbox{\rm mod}\,#1)}
\def\Res{\mathrm{Res}\,}
\newcommand\bi{\vspace{-4pt}\begin{itemize}\itemsep -3pt plus 1pt minus 1pt}
\newcommand\ei{\end{itemize}\vspace{-4pt}}
\newcommand\bn{\begin{enumerate}}
\newcommand\en{\end{enumerate}}
\newcommand\cC{\mathcal{C}}

\newcommand\cE{\mathcal{E}}
\newcommand\cZ{\mathcal{Z}}
\newcommand\cV{\mathcal{V}}

\newcommand\RR{\mathbf{R}}
\newcommand\CC{\mathbf{C}}
\newcommand\QQ{\mathbf{Q}}
\newcommand\ZZ{\mathbf{Z}}
\newcommand\Disc{\mathrm{Disc}\,}
\def\sign#1{{\rm sign}\,\bigl( #1 \bigr)}
\newcommand\Sum{\mathop{\sum}\limits}
\def\S#1{{\mathbf S_{#1}}}
\def\abs#1{\left \vert #1 \right \vert}
\def\frac#1#2{{\textstyle{{#1} \overwithdelims.. {#2}}}}
\def\Frac#1#2{{\displaystyle{{#1} \overwithdelims.. {#2}}}}
\renewcommand\phi{\varphi}
\newtheorem{theo}{Theorem}
\newtheorem{lemma}[theo]{Lemma}
\newtheorem{proposition}[theo]{Proposition}
\newtheorem{corollary}[theo]{Corollary}

\begin{document}
%

\title{Computing Chebyshev knot diagrams}
\numberofauthors{3}
\author{
\alignauthor
Pierre-Vincent Koseleff\\
       \affaddr{INRIA, Paris-Rocquencourt, SALSA Project}\\
       \affaddr{UPMC-Université Paris 6}\\
       \affaddr{CNRS, UMR 7606, LIP6}\\
       \email{koseleff@math.jussieu.fr}
\alignauthor
Daniel Pecker\\
       \affaddr{UPMC-Université Paris 6}\\
       \email{pecker@math.jussieu.fr}
\alignauthor
Fabrice Rouillier\\
       \affaddr{INRIA, Paris-Rocquencourt, SALSA Project}\\
       \affaddr{UPMC-Université Paris 6}\\
       \affaddr{CNRS, UMR 7606, LIP6}\\
       \email{fabrice.rouillier@inria.fr}
}
\maketitle
\begin{abstract}
A Chebyshev curve  $\cC(a,b,c,\phi)$ has
a parametrization of the form $ x(t)=T_a(t); \  y(t)=T_b(t) ; \
z(t)= T_c(t + \phi), $ where $a,b,c$ are integers, $T_n(t)$ is the
Chebyshev polynomial of degree $n$ and $\phi \in \RR$.
When $\cC(a,b,c,\phi)$ has no double points, it defines a polynomial
knot. We determine all possible knots when
$a$, $b$ and $c$ are given.
\end{abstract}
\keywords{Zero dimensional systems, Chebyshev curves, Lissajous knots,
polynomial knots, factorization of Chebyshev polynomials, minimal polynomial}
\section{Introduction}\label{intro}
It is known that every knot may be obtained from a polynomial embedding
$\RR \to \RR^3$
(\cite{Va,DOS}).

Chebyshev knots are polynomial analogue to Lissajous knots that have been studied by many authors
(see \cite{BDHZ,BHJS,HZ,JP,La2}). All knots are not Lissajous (for example the trefoil and the figure-eight knot).
In \cite{KP3}, it is proved that any knot $K \subset \RR^3$ is a Chebyshev knot,
that is to say there exist positive integers $a$, $b$,   $c$ and a real $\phi$
such that $K$ is isotopic to the curve
$$
\cC(a,b,c,\phi) : x=T_a(t), \, y=T_b(t), \, z=T_c(t+\phi),
$$
where $T_n$ is the Chebyshev polynomial of degree $n$.
This is our motivation for the study of curves $\cC(a,b,c,\phi)$, $\phi \in \RR$.
\pn
In \cite{KP3}, the proof uses theorems on braids by Hoste, Zirbel and Lamm (\cite{HZ,La2}),
and a density argument (Kronecker theorem).

In \cite{KPR}, we developed an effective method
to enumerate all the knots $\cC(a,b,c,\phi), \phi \in \RR$ where $a=3$ or $a=4$, $a$ and
$b$ coprime.
This method was based on continued fraction expansion theory in order to get the minimal
$b$, on resultant computations in order to determine the critical values $\phi$ for which
$\cC(a,b,c,\phi)$ is singular, and on multi-precision interval arithmetic to determine the knot type
of $\cC(a,b,c,\phi)$.
Our goal was to give an exhaustive list of the minimal parametrizations for the first
95 rational knots with less than 10 crossings. We obtained in \cite{KPR} almost every minimal
parametrizations. For 6 of these knots, we know the minimal $b$ and we know that the corresponding $c$ must be $>300$.
\pn In this paper, we develop a more efficient algorithm.
It will give the parametrization of the 6 missing knots in \cite{KPR} and allows to compute all diagrams corresponding to $\cC(a,b,c,\phi)$, $\phi \in \RR$.
One motivation is first to achieve the exhaustive list of certified minimal Chebyshev parametrizations for the first 95 rational knots.
Another is to provide a certified tool for the study of polynomial curves topology.
\pn
Let us first recall some basic facts about knots.
\subsection*{Knot diagrams}\label{diagrams}
The projection of the Chebyshev space curve $\cC(a,b,c,\phi)$
on the $(x,y)$-plane is the Chebyshev curve $\cC(a,b): x=T_a(t), \, y=T_b(t)$.
If $a$ and $b$ are coprime integers, the curve $\cC(a,b,c,\phi)$ is singular
if and only if it has some double points.
It is convenient to consider the polynomials in $\QQ[s,t,\phi]$
\[
P_n= \Frac{T_n(t)-T_n(s)}{t-s},\
Q_n= \Frac{T_n(t+\phi)-T_n(s+\phi)}{t-s}. \label{PQ}
\]
We see that $\cC(a,b,c,\phi)$ is a knot iff $\{(s,t), \, P_a(s,t)=P_b(s,t)=Q_c(s,t,\phi)=0\}$ is empty.
\pn
We shall study the diagram of the curve $\cC(a,b,c,\phi)$, that is to say the plane projection
$\cC(a,b)$ onto the $(x,y)$-plane and the nature (under/over) of the crossings over the double points
of $\cC(a,b)$.
There are two cases of crossing: the right twist and the left twist (see \cite{Mu} and Figure \ref{signf})).
\begin{figure}[th]
\begin{center}
\begin{tabular}{ccc}
{\scalebox{.1}{\includegraphics{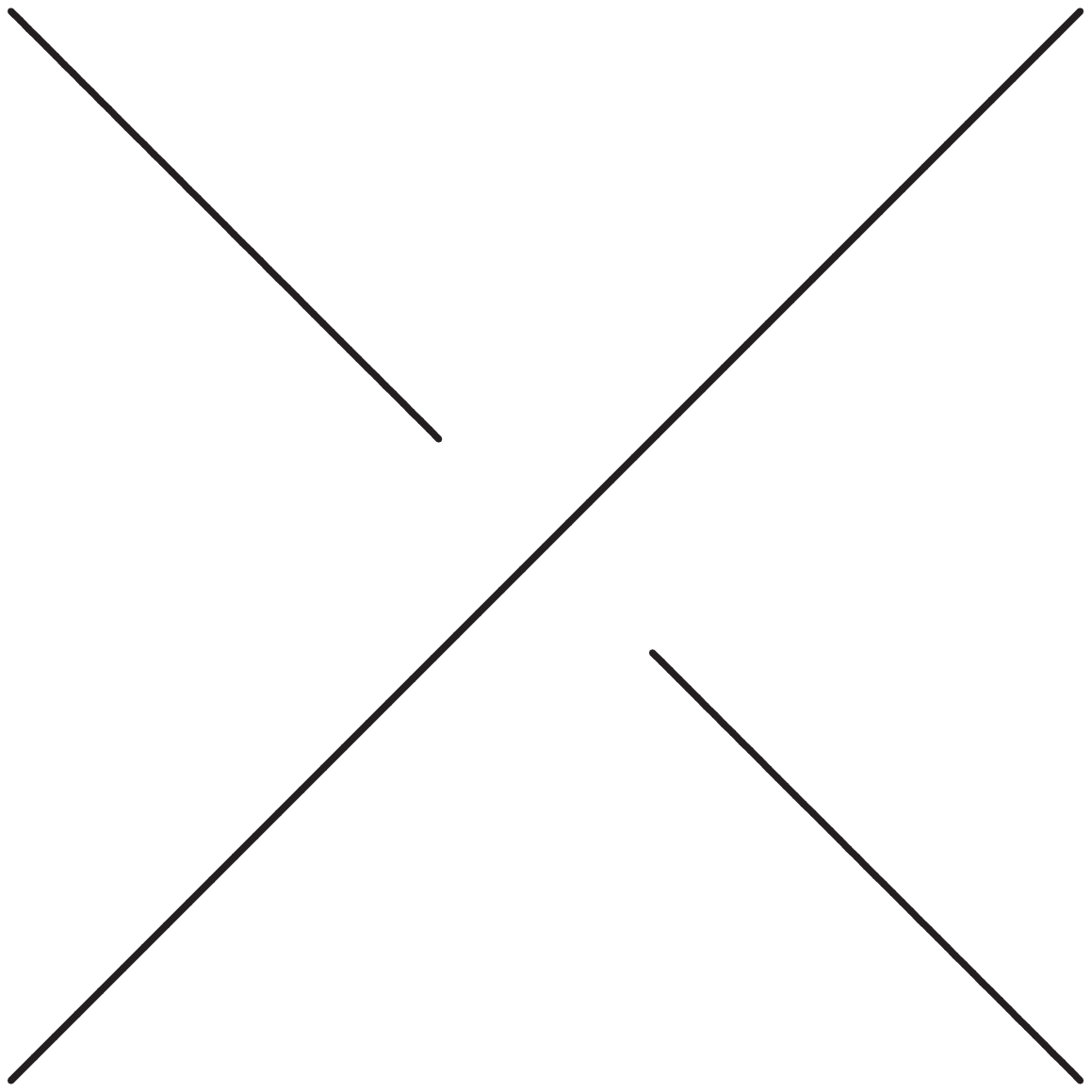}}} &\quad&
{\scalebox{.1}{\includegraphics{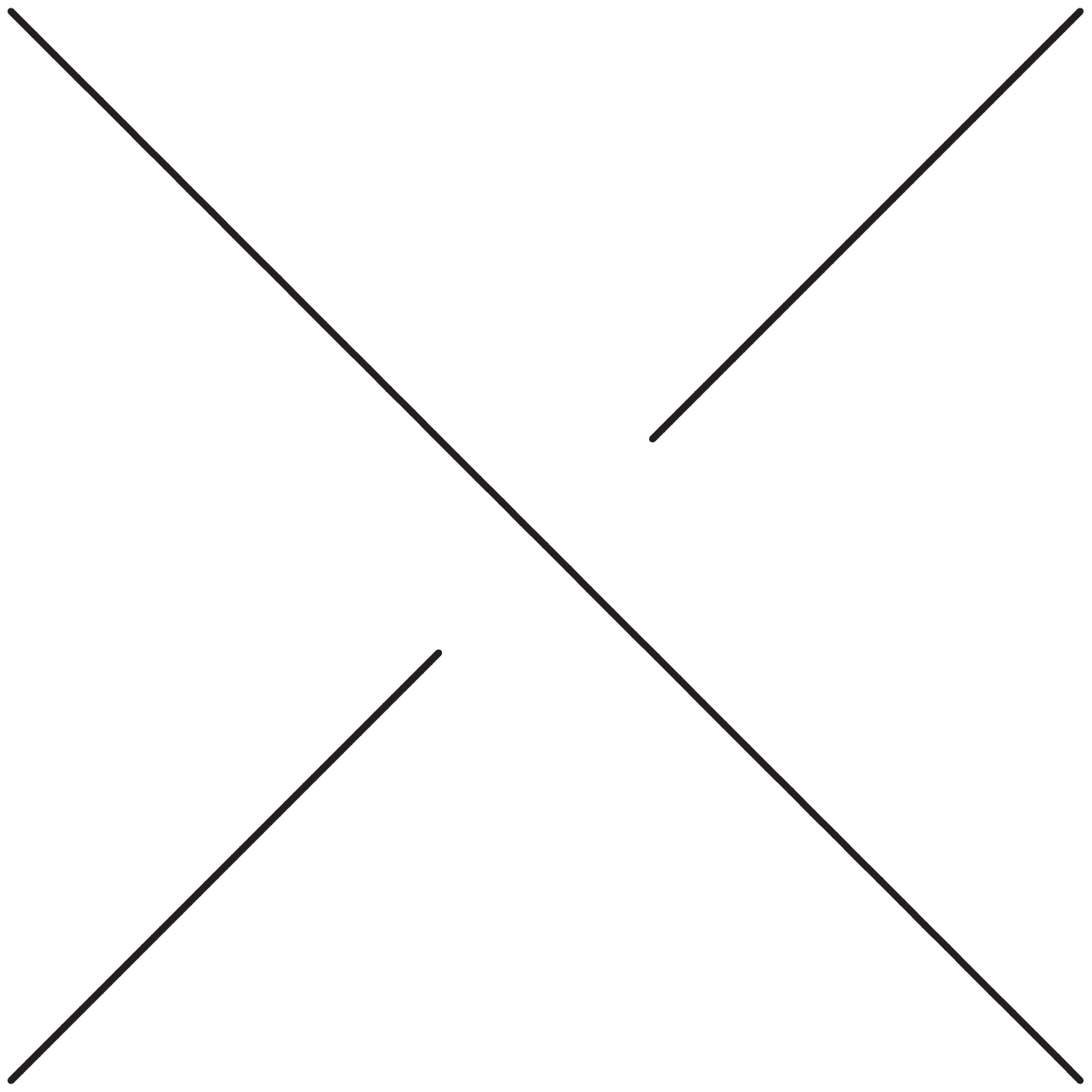}}}
\end{tabular}
\caption{The right twist and the left twist\label{signf}}
\end{center}
\vspace{-10pt}
\end{figure}
In \cite{KPR}, we showed that the nature of the crossing over the double point
$A_{\alpha,\beta}$ corresponding to parameters
($t= \cos(\alpha+\beta), s=\cos(\alpha-\beta)$, $\alpha = \frac{i\pi}a$,
$\beta=\frac{j\pi}b$), is given by the sign of
\begin{equation}
D(s,t,\phi) = Q_c(s,t,\phi) P_{b-a}(s,t,\phi). \label{D}
\end{equation}
$D(s,t,\phi)>0$ if and only if the crossing is a right twist.
\pn
Note that the crossing points of the Chebyshev curve $\cC(a,b): x=T_a(t), \, y=T_b(t)$ lie on
the $(b-1)$ vertical lines $T'_b(x)=0$ and on the $(a-1)$ horizontal lines $T'_a(y)=0$.
We can represent the knot diagram of $\cC(a,b,c,r)$ by a billiard trajectory
(see \cite{KP3}), which is a pure combinatorial object. As an example, consider the knots
$\overline{5}_2 = \cC(4,5,7,0)$,
${5}_2 = \cC(5,6,7,0)$,
$\overline{4}_1 = \cC(3,5,7,0)$.
\begin{figure}[th]
\begin{tabular}{ccc}
\epsfig{file=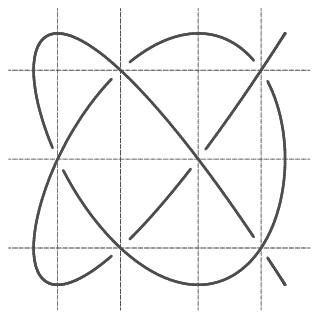,height=2.5cm,width=2.5cm}&
\epsfig{file=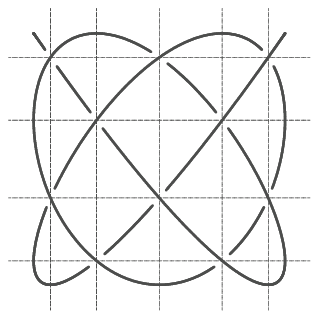,height=2.5cm,width=2.5cm}&
\epsfig{file=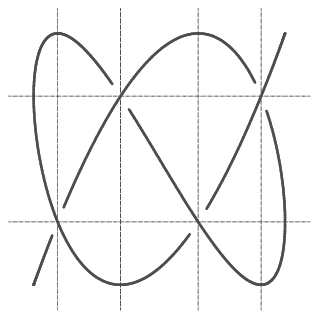,height=2.5cm,width=2.5cm}\\
$\overline{5}_2$&${5}_2$&$\overline{4}_1$
\end{tabular}
\end{figure}
We can represent their diagrams by the billiard trajectories in Figure \ref{bt}.
\begin{figure}[th]
\begin{tabular}{ccc}
\epsfig{file=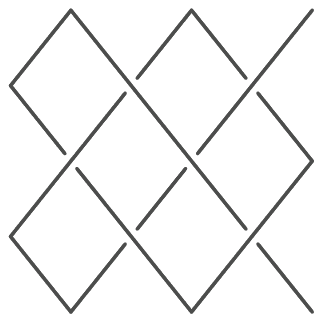,height=2.5cm,width=2.5cm}&
\epsfig{file=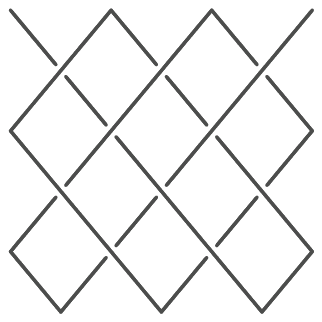,height=2.5cm,width=2.5cm}&
\epsfig{file=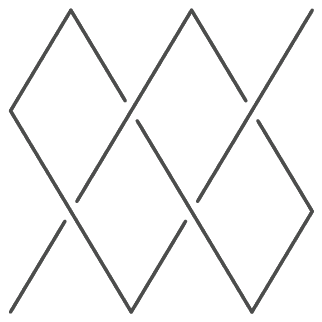,height=2.5cm,width=2.5cm}
\end{tabular}
\caption{Billiard trajectories}\label{bt}
\end{figure}
\pn
When $a=3$ or $a=4$, we obtain the diagrams in the Conway normal form. In this case, the knot is rational
and can be identified very easily by its Schubert fraction (see \cite{Mu,KP4}).
When $b>a\geq 5$, the problem of classification is much more difficult.
Nevertheless, the knowledge of the diagrams allows the computation of all the classical invariants,
like the Conway, Alexander and Jones polynomials (see \cite{Mu}).
\subsection*{Summary}
Our goal is to compute all diagrams of $\cC(a,b,c,\phi)$, where $a$, $b$, $c$ are given integers and $\phi \in \RR$.
From the algorithmic point of view, the description of the Chebyshev knots is strongly connected to the resolution of:
\begin{equation}
\cV_{a,b,c}=\{ P_a(s,t)=0, \, P_b(s,t)=0, \, Q_c(s,t,\phi)=0 \}. \label{theeq}
\end{equation}
We first want to determine the set $\cZ_{a,b,c}$ of critical values $\phi$ for which the curve
$\cC(a,b,c,\phi)$  is singular.
Because $\deg_{\phi} Q_n = n-1$ and the leading term of $Q_n$ is $2^{n-1} n \phi^{n-1}$, we showed
in \cite{KPR}, that $\cV_{a,b,c}$
is 0-dimensional and has at most $(a-1)(b-1)(c-1)$ points. We deduced that
$\abs{\cZ_{a,b,c}} \leq \frac 12 (a-1)(b-1)(c-1)$.
\pn
Let $\cZ_{a,b,c} = \{ \phi_1, \ldots , \phi_n\}$.
The type of the knot $\cC(a,b,c,\phi)$ is given by its diagram which is constant when $\phi$ is in $(\phi_i,\phi_{i+1})$, because
the crossings do not change in this interval.
In order to get all possible knots $\cC(a,b,c,\phi)$, we only need sample points $r_i$ in each $(\phi_i, \phi_{i+1})$ and
to compute the diagram of $\cC(a,b,c,r_i)$.
\pn
We can determine a polynomial $R_{a,b,c} \in \ZZ[\phi]$ such that
$\cZ_{a,b,c} = Z(R)$. It can be defined by
$\langle R \rangle = \langle P_a,P_b, Q_c \rangle \bigcap \QQ[\phi]$ and may be obtained
with Gröbner bases (\cite{CLOS}).
In \cite{KPR}, we optimized the computation by an ad-hoc elimination based on the properties of the curves for $a=3$ or
$a=4$.
Gröbner bases could fully be substituted by some resultant computations in $\ZZ[X,\phi]$,
the systems being generic enough. However, this leads to solve systems of very high degree.
%
\pn%
In the present paper we
decompose the system by working on some (real cyclotomic) extension fields. We show that the system (\ref{theeq})
is equivalent to the resolution of $\frac 12 (a-1)(b-1)\pent{c}2$ second-degree polynomials with coefficients in
$\QQ(\cos\frac\pi{a},\cos\frac\pi{b},\cos\frac\pi{c})$. This result is deduced from geometric properties of the implicit
Chebyshev curves.
\pn
We show some properties of these extensions that allow to simplify the computations.
We can represent the coefficients of the polynomials by intervals and certify the resolution.
We then easily and independently obtain the roots of the second-order polynomials and the main
difficulty becomes to compare them.
A formal method would consist in computing their minimal polynomials over $\QQ$, which is equivalent to
the resolution of (\ref{theeq}).
We use multi-precision interval arithmetic for coding the algebraic numbers $\cos\frac{k\pi}n$ as well as the solutions $\phi$ we get.
If the two intervals are disjoint, the roots are distinct. If not, we can certify whether the resultant of the two second-order polynomials equals 0 or not by Euclidean division.
\pn
In section \ref{cheb},
we first describe the Chebyshev polynomials and the link between their factorizations
and the minimal polynomials of $\cos\frac{k\pi}n$.
This allows us to represent efficiently the elements of
$\QQ(\cos\frac\pi{a},\cos\frac\pi{b},\cos\frac\pi{c})$.
Along the way, we give an explicit factorization of the Chebyshev polynomials.
\pn
In section \ref{curves},
we recall the definition of Lissajous curves and we give their implicit equations.
We study the affine implicit Chebyshev curves $T_n(x)=T_m(y)$ and show
that they have $\pent{(n,m)}2 +1$ irreducible components, $\pent{(n,m)-1}{2}$ being Lissajous curves.
\pn
This allows us to deduce an explicit factorization of $R_{a,b,c}$ as the product of second-degree polynomials
$P_{\alpha,\beta,\gamma}$,
 in section \ref{cv}.
\pn
We show  in section \ref{ccv}, how to obtain $\cZ_{a,b,c}$, the set of roots of $R_{a,b,c}$, with their multiplicities.
The general algorithm is described in \ref{algo}.
This allows us to sample all Chebyshev knots $\cC(a,b,c,\phi)$, $\phi \in \RR$, by choosing
a rational number $r$ in each component of $\RR - \cZ_{a,b,c}$.
\pn
In section \ref{experiments}, we find an exhaustive and complete list of the minimal parametrization for the first 95 rational knots.
The worst case appears with the knot $10_{33} = \cC(4,13,856,1/328)$, with $\deg R_{a,b,c}= 15390$. We discuss the efficiency of our algorithms and compare with those of \cite{KPR}.
\section{Chebyshev polynomials}\label{cheb}
Chebyshev polynomials and their algebraic properties play a
central role here. The curves we will study are defined by Chebyshev polynomials. The algebraic
extensions we will consider are spanned by their roots and we need to know their factors.
In this section we recall some classical properties of Chebyshev polynomials.
We will also show the link between their effective factorization in $\QQ[t]$
and the minimal polynomial of $\cos\frac{k\pi}n$.
\pn
The Chebyshev polynomials of the {\em first kind} are defined
by the second-order linear recurrence
\[
T_0 = 1, \, T_1 = t, \, T_{n+1} = 2tT_n - T_{n-1}.\label{rl1}
\]
$T_n \in \ZZ[t]$ and
satisfies the identity $T_n(\cos\theta) = \cos n\theta$, and more generally
$T_n \circ T_m = T_{nm}$. We 
have
$$
T_n = 2^{n-1}\Prod_{k=0}^{n-1} (t-\cos\frac{(2k+1)\pi}{2n}).
$$
Let $V_n$ be the Chebyshev
polynomials of the {\em second kind} defined by the second-order linear recurrence (the same as
in (\ref{rl1}))
$$
V_0 = 0, \, V_1 = 1, \, V_{n+1} = 2tV_n - V_{n-1}.
$$
$V_n \in \ZZ[t]$ and
satisfies $V_n(\cos\theta) = \Frac{\sin n\theta}{\sin\theta}$.
We have
$$
V_n = 2^{n-1} \Prod_{k=1}^{n-1} (t-\cos\frac{k\pi}n),
$$
and therefore $V_d \vert V_n$ when $d \vert n$.
Let us summarize some useful results in the following
\begin{lemma} \label{pp}
We have the following properties:
\bi
\item $T'_n(t)=0 \Rightarrow T_n(t)=\pm 1$
\item $T_n(t)=\pm 1 \Rightarrow T'_n(t)=0$ or  $t = \pm 1$.
\item $T_n(t)=y$ has $n$ real solutions iff $\abs y <1$.
\item $T_n(t)=1$ has $\pent n2$ real solutions.
\item $T_n(t)=-1$ has $\pent{n-1}2$ real solutions.
\ei
\end{lemma}
\begin{proof}
From $T'_n = n V_n$, we deduce that $t\mapsto T_n(t)$ is monotonic when $\abs t \geq \cos\frac{\pi}n$,
that $T_n$ has $n-1$ local extrema for $t_k = \cos\frac{k\pi}n$ and $T_n(t_k)=(-1)^k$.
\end{proof}
\subsection*{Minimal Polynomial of $\mathbf{\cos \frac{k\pi}n}$}
Let $\zeta_n = e^{\frac{2i\pi}n}$. It is well known (\cite{WZ}) that
the degree of $\QQ(\zeta_n)$ is $\phi(n)$ where $\phi$ is the Euler function.
$\QQ(\cos\frac{2\pi}n) = \QQ(\zeta_n) \bigcap \RR$ and
the minimal polynomial over $\QQ$ of $\cos\frac{2\pi}n$ has degree
$\frac 12 \phi(n)$ when $n>1$. Its roots are $\cos\frac{2k\pi}n$ where
$k$ is coprime with $n$. Consequently, the minimal polynomial
$M_n$ of $\cos\frac{\pi}n$ has degree $\frac 12 \phi(2n)$, when $n>1$.
Its roots are $t_k = \cos\frac{k\pi}n$ where $(k,n)=1$, and $k$ is odd.
$M_n(-t)$ is the minimal polynomial of $\cos\frac{2\pi}n$. The leading
coefficient of $M_n$ is $2^{\phi(2n)/2}$.
\pn
{\bf Remark.} $\cos\frac{k\pi}n \in \QQ$ iff $\frac 12 \phi(2n)=1$ or $n=1$, that is
$n=1, 2,3$. In this case we get $2\cos\frac{k\pi}n \in \ZZ$.
\pn
We deduce the following
\begin{proposition}\label{cminf}
Let $P_n$ be defined by $P_0=1$, $P_1=2t-1$, $P_{n+1} = 2t P_n + P_{n-1}$. Then we have
$(-1)^n P_n(-T_2) = V_{2n+1}$ and
\[P_n = \Prod_{d\vert 2n+1} M_d\label{minf} \]
\end{proposition}
\begin{proof}
We have  $P_0(-T_2) = V_1$, $P_1(-T_2) = -2T_2 -1 = -V_3$. The sequences $V_{2n+1}$ and $(-1)^n P_n(-T_2)$ satisfy
the same recurrence formula: $V_{2n+3} = 2T_2 V_{2n+1} - V_{2n-1}$.
Let $d=2m+1$ be a divisor of $2n+1$ and consider $t=\cos\frac{\pi}d = -\cos 2 \frac{m\pi}{2m+1}$. We have
$(-1)^n P_n(t) = V_{2n+1}(\cos \frac{m\pi}{2m+1}) = 0$. Thus $M_d \vert P_n$ and we conclude
using the fact that \\
$\Sum_{d\vert 2n+1} \deg M_d = 1 + \frac 12 \Sum_{d \vert 2n+1, d>1} \phi(2d) = n = \deg P_n.$
\end{proof}
\begin{lemma}\label{mine}
We have $M_{2^k m} = M_{m}(T_{2^k})$ if $m$ is odd.
\end{lemma}
\begin{proof}
We have $M_{m}\circ T_{2^k} (\cos\frac{\pi}{2^km}) = 0$ and $(2^k,m)=1$  so $M_{2^km} \vert M_{m}(T_{2^k})$.
We conclude since $M_{2^k m}$ and $M_{m}(T_{2^k})$ have same leading term.
\end{proof}
The relations between the minimal polynomial of $\cos\frac{2\pi}n$ and the factorization
of $T_{\pent n2+1}-T_{\pent n2}$ is known (\cite{WZ}).
Formula (\ref{minf}) together with Lemma \ref{mine} give also an algorithm to compute $M_n$.
\pn
The number of factors of $T_n$ is known (\cite{Hs}).
We give here the relation between the Chebyshev polynomials $T_n$ and $V_n$ and the polynomials $M_n$.
\begin{proposition}{\bf Factorization of $T_n$ and $V_n$.}\\
We have the following factorizations in irreducible factors
\begin{eqnarray*}
V_{2^k(2m+1)} &=& \prod_{d\vert 2m+1} \left (\prod_{i=1}^k M_{d}(T_{2^i}) \right )\cdot M_d(t) M_d(-t)\\
T_{2^k (2m+1)} &=& \Frac 12 \prod_{d\vert 2m+1} M_d(T_{2^{k+1}})
\end{eqnarray*}
where $M_n$ is the minimal polynomial of $\cos\frac\pi n$.
\end{proposition}
\begin{proof}
The factorization of $V_n$ is obtained by comparing its roots with those of $M_d(\pm t)$, when $d\vert n$.
Let $d$ be an odd divisor of $n$. We write $n=2^k\cdot d_1\cdot d$, where $d_1$ is odd.
$\cos\frac{d_1\pi}{2n}=\cos\frac{\pi}{2^{k+1}d}$ is a root of $T_n$ so
$M_{2^{k+1}d} \vert T_n$. We deduce the factorization by comparing the leading terms.
\end{proof}
\section{Chebyshev and Lissajous curves}\label{curves}
The following proposition will explain the notions of Lissajous and Chebyshev
curves.
\begin{proposition} \label{equa}
The parametric curve
$$
\cC: x=\cos( at), \, y = \cos(bt+\phi), \, t \in \CC,
$$
where $a,b$ are coprime integers ($a$ odd) and $\phi \in \RR$ admits the equation
\[
T_b(x)^2 + T_a(y)^2 - 2 \cos(a\phi)T_b(x)T_a(y) - \sin^2(a\phi)=0.\label{eqE}
\]
\begin{enumerate}
\item If $a\phi \not = k \pi$, this equation is irreducible. $\cC$ is called a Lissajous
curve. Its real part is 1-1 parametrized for $t \in [0,2\pi]$.
\item If $a\phi = k \pi$, this equation is equivalent to $T_b(x)=(-1)^k T_a(y)$.
$\cC$ is called a Chebyshev curve. It can be parametrized by $x= T_a(t), \, y=(-1)^k T_b(t)$.
\end{enumerate}
\end{proposition}
\begin{proof}
Let $(x,y) \in \cC$. We have $T_b(x)=\cos( bat), \, T_a(y) = \cos(bat + a\phi)$.
Let $\lambda = a\phi, \, \theta=abt$. We get
$T_a(y) = \cos(\theta+\lambda)$ so
$(1-\cos^2 \theta)\sin^2\lambda = (\cos\theta \cos\lambda - T_a(y))^2$, that is
$
(1-T_a(x))^2 \sin^2\lambda = (T_a(x)\cos\lambda - T_a(y))^2,
$
and we deduce our Equation (\ref{eqE}).
\pn
Conversely, suppose that $(x,y)$ satisfies (\ref{eqE}). Let $x=\cos(at)$ where $t \in \CC$. We also have
$x=\cos a(t+\frac{2k\pi}a)$. We have $T_b(x)= \cos\theta$. $A=T_a(y)$ is solution of the second-degree
equation
$$
A^2 - 2\cos(a\phi)\cos\theta A - \sin^2(a\phi)=0.
$$
Consequently, we get 
$T_a(y) = \cos(\theta \pm a\phi) = T_a(\cos(\pm bt + \phi))$. We deduce
that $y = \cos(\pm bt + \phi + \frac{2h\pi}a)$, $h \in \ZZ$.
Changing $t$ by $-t$, we can suppose that
$$
x=\cos at, \, y = \cos(bt+\phi +\frac{2h\pi}{a}).
$$
By choosing $k$ such that $kb+h\equiv 0 \Mod a$, we get
$
x=\cos at', \, y = \cos(bt'+\phi),
$
where $t'=t+\frac{2k\pi}a$.
\pn
If $a\phi \not \equiv 0 \Mod \pi$. Suppose that Equation (\ref{eqE}) factors in
$P(x,y) Q(x,y)$. We can suppose, for analyticity reasons, that $P(\cos (at) , \cos (bt+\phi))=0$, for $t\in \CC$.
The curve $\cC$ intersects the line $y=0$ in $2b$ distinct points so $\deg_x P \geq 2b$. Similarly,
$\deg_y P \geq 2a$ so that $Q$ is a constant which proves that the equation is irreducible.
\pn If $\cos a\phi = (-1)^k$, the equation becomes $T_b(x)-(-1)^k T_a(y)=0$. In this case
the curve admits the announced parametrization (see \cite{Fi,KP3} for more details).
\end{proof}
{\bf Remark.}
If $a=b=1$, we obtain the Lissajous ellipses. They are the first curves studied by Lissajous (\cite{Li}).
Let $\mu \not \equiv 0 \Mod{\pi}.$ The curve
$$\cE_{\mu}: x^2+y^2- 2 \cos(\mu)xy - \sin^2(\mu)=0$$
is an ellipse inscribed in the square $[-1,1]^2$. It admits the
parametrization $x=\cos t, \, y=\cos (t +\mu)$.
\pn
The following notation  will be useful.
Let $E_{\mu}(x,y) = x^2+y^2- 2 \cos(\mu)xy - \sin^2(\mu)$ when $\mu \not\equiv 0 \Mod{\pi}$ and
$E_0 = x-y$, $E_{\pi} = x+y$. The Equation (\ref{eqE}) is equivalent to
$E_{a\phi} (T_b(x),T_a(y))=0$.
This shows that the real part of the curve $\cC$ (Equation (\ref{eqE})) is inscribed in the square
$[-1,1]^2$.
Using Proposition \ref{equa} we recover the classical following result.
\begin{corollary}
The Lissajous curve $x=\cos(at), \, y=\cos(bt+\phi)$, ($a\phi\not \equiv 0 \Mod \pi$)
has $2ab - a - b$ singular points which are real double points.
\end{corollary}
\begin{proof}
singular points of $\cC$ satisfy Equation (\ref{eqE}) and the system
\begin{eqnarray*}
T'_b(x) (T_b(x) - T_a(y) \cos a\phi)=0, \\
T'_a(y) (T_a(y) - T_b(x) \cos a\phi)=0.
\end{eqnarray*}
Suppose that $T'_b(x)=T'_a(y)=0$ then $T^2_a(y) = T^2_b(x)=1$ and Equation \ref{eqE} is not satisfied.
Suppose that $T_b(x) - T_a(y) \cos a\phi=T_a(y) - T_b(x) \cos a\phi=0$, then $T_b(x)=T_a(x)=0$ and
Equation \ref{eqE} is not satisfied.
We thus have either $T'_b(x)=0$ and $T_a(y) - T_b(x) \cos a\phi=0$ that gives $(b-1)\times a$ real
points because of the classical properties of Chebyshev polynomials, or
$T'_a(y)=0$ and $T_b(x) - T_a(y) \cos a\phi=0$ that gives $b\times (a-1)$ real double points.
\end{proof}
{\bf Remark.}
The study of the double points of Lissajous curves is classical (see \cite{BHJS} for  their parameters values).
The study of the double points of Chebyshev curves is simpler (see \cite{KP3}).
\begin{corollary}
The affine implicit curve $T_n(x)=T_m(y)$ has $\pent{n-1}2\pent{m-1}2 + \pent{n}2\pent{m}2$ singular points that
are real double points.
\end{corollary}
\begin{proof}
The singular points satisfy either $T_n(x)=T_m(y)=1$ or $T_n(x)=T_m(y)=-1$ and we conclude using Lemma \ref{pp}.
\end{proof}
\begin{theo}{\bf Factorization of $\mathbf{T_n(x)-T_m(y)}$.}\label{fact}\\%
Let $m=ad$, $n=bd$, $(a,b)=1$ and $a$ odd. We have the factorization
$$
T_n(x)-T_m(y) = 2^{d-1} \Prod_{k=0}^{\pent d2} C_k(x,y)
$$
where
$$
C_k(x,y)=E_{\frac{2ak\pi}d}(T_b(x),T_a(y))
$$
is the irreducible equation of the curve
$\cC_k : x=\cos (at), \, y = \cos (bt+\frac{2k\pi}{d})$, given in Proposition
\ref{equa}.
\end{theo}
\begin{proof}
Let $\cC$ be the curve $T_n(x)=T_m(y)$. We easily get $\cC_k \subset \cC$ and
$\cC_k \not = \cC_{k'}$.
When $k=0$, $\cC_0$ admits the equation $T_b(x)-T_a(x)=0$.
If $2k=d$, $\cC_k$ admits the equation $T_b(x)+T_a(y)=0$.
In the other cases, the dominant term in $x$ of $C_k$ is $2^{2b-2}x^{2b}$.
If $d$ is even, we deduce that the dominant term of
$\Prod_{k=0}^{\pent d2} C_k(x,y)$ is $2^{(b-1)d}x^{2n}$ and we get our result in this case.
If $d$ is odd, we get the same result.
\end{proof}
\begin{corollary}
Let $d=\gcd(a,b)$. The curve $T_b(x)=T_a(y)$ has $\pent d2 +1$ components.
$\pent{d-1}2$ of them are Lissajous curves.
\end{corollary}
\begin{figure}[th]
\begin{tabular}{cc}
\epsfig{file=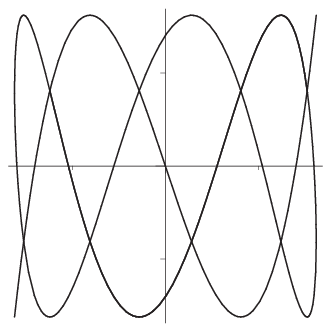,height=3.5cm,width=3.5cm}&
\epsfig{file=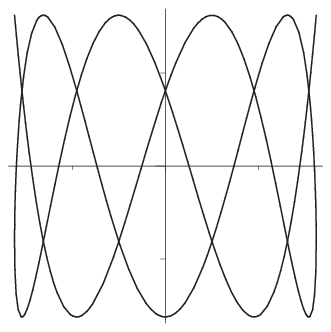,height=3.5cm,width=3.5cm}\\
{\small $T_{9}(x)=T_3(y)$}&{\small $T_{10}(x)=T_3(y)$}\\
\epsfig{file=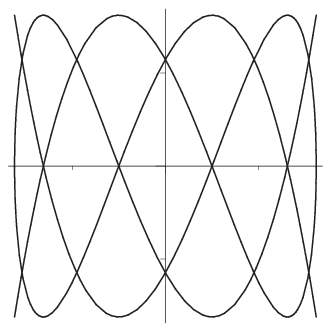,height=3.5cm,width=3.5cm}&
\epsfig{file=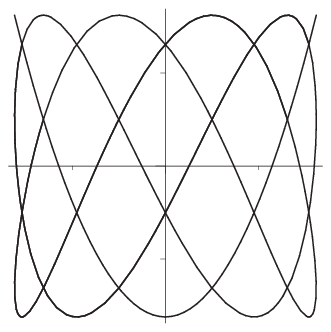,height=3.5cm,width=3.5cm}\\
{\small $T_{10}(x)=T_4(y)$}&{\small $T_{10}(x)=T_5(y)$}
\end{tabular}
\caption{Implicit Chebyshev curves}
\label{Td}
\end{figure}
Theorem \ref{fact} is particularly  interesting when $m=n=d$ and $a=b=1$.
In this case the curve $T_n(x)=T_n(y)$ is a union of ellipses and some lines.
It will be useful for the determination of the double points of Chebyshev space curves.
We have
\[
\Frac{T_n(t)-T_n(s)}{t-s} = 2^{n-1} \Prod_{k=1}^{\pent{n}{2}} E_{\frac{2k\pi}n}(s,t).\label{fell}
\]
The curve  $\Frac{T_n(t)-T_n(s)}{t-s}=0$ has $\pent n2$ irreducible components.
Note that $\mathcal{E}_{\frac{2k\pi}n}$ and $\mathcal{E}_{\frac{2l\pi}m}$ intersect at the point
$(t,s) = (\cos(\frac{k\pi}{n}+\frac{l\pi}{m}),\cos(\frac{k\pi}{n}-\frac{l\pi}{m}))$ and its symmetric
with respect to the lines $s=-t$ and $s=t$.
We recover the parametrization of the double points of $x=T_a(t), \, y=T_b(t)$ that will be
very useful for the description of Chebyshev space curves.
\begin{proposition}[\cite{KP3,KPR}]\label{dp}
Let $a$ and $b$ are nonnegative coprime integers, a being odd. Let
the Chebyshev curve $\cC$ be defined by  $ x= T_a(t), \  y=T_b(t).$
The pairs $(t,s)$ giving a crossing point are
$$t=\cos (\frac{j\pi}b+\frac{i\pi}a), \ s=\cos(\frac{j\pi}b-\frac{i\pi}a)$$
where $1\leq i \leq \frac 12 (a-1)$, $1\leq j \leq b-1$.
\end{proposition}
\begin{figure}[th]
\begin{tabular}{ccc}
\epsfig{file=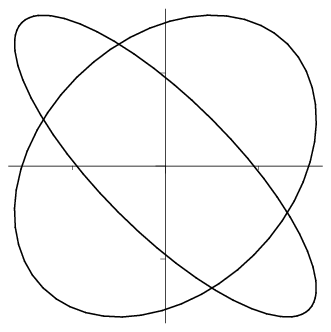,height=3.5cm,width=3.5cm}&&
\epsfig{file=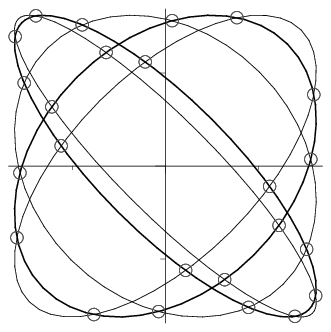,height=3.5cm,width=3.5cm}\\
{\small $\Frac{T_5(t)-T_5(s)}{t-s}=0$}&&
{\small
$\left \{
\begin{array}{c}
\Frac{T_7(t)-T_7(s)}{t-s}=0\\
\Frac{T_5(t)-T_5(s)}{t-s}=0
\end{array}\right .$}
\end{tabular}
\caption{Double points in the parameters space}
\label{T7}
\vspace{-10pt}
\end{figure}
\section{Critical values}\label{cv}
A polynomial $R_{a,b,c} \in \ZZ[\phi]$ for which
$\cZ_{a,b,c} = Z(R)$ can be defined by
$\langle R \rangle = \langle P_a,P_b, Q_c \rangle \bigcap \QQ[\phi]$ and may be obtained
with Gröbner bases (\cite{CLOS}).
\pn
{\bf Example.}
When $a=3$, $b=4$, $c=5$, we find that \\
$R_{a,b,c} =
 \left( 80\,{\phi}^{4}+60\,{\phi}^{2}-1 \right) \cdot $\\
 \hbox{}\hfill
$\left( 6400\,{\phi}^{8}-3200\,{\phi}^{6}+560\,{\phi}^{4}-80\,{\phi}^{2}+1 \right)$.\\
There are exactly 6 critical values that are symmetrical about the origin.
\begin{figure}[th]
\begin{tabular}{ccc}
\epsfig{file=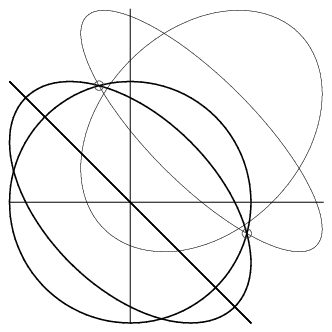,height=3.5cm,width=3.5cm}&&
\epsfig{file=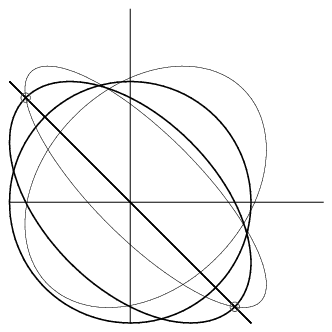,height=3.5cm,width=3.5cm}\\
{\small $\phi=.590$}&&{\small $\phi=.128$}\\
\epsfig{file=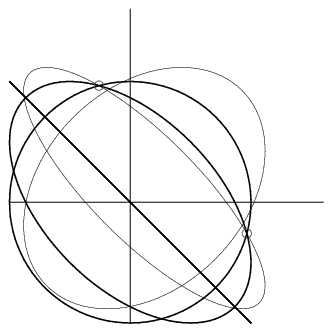,height=3.5cm,width=3.5cm}&&
\epsfig{file=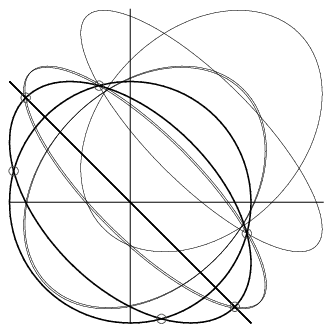,height=3.5cm,width=3.5cm}\\
{\small $\phi=.117$}&&{\small $\phi=.117,.128,.590$}\\
\end{tabular}
\caption{$P_3=0, P_4=0, Q_5=0$}
\label{T345}
\end{figure}
For these values of $\phi$, the curve $Q_5(s,t,\phi)=0$, which is translated from
the curve $P_5(s,t)=0$ by the vector $(\phi,\phi)$, meets the points
$\{P_3=0, P_4=0\}$.
\pn
In this part, we use use the properties of Chebyshev curves obtained in section \ref{curves}.
We give an explicit formula for the polynomial $R_{a,b,c}$ as a product of
of univariate polynomials of degree 1 or 2 with coefficients in $\QQ(\cos\frac\pi{a},\cos\frac\pi{b},\cos\frac\pi{c})$.
\begin{proposition}
Let $a,b$ be nonnegative coprime integers and $c$ be an integer. Suppose that
$a$ is odd.
Let $R_{a,b,c}(\phi)$ be the polynomial
\begin{equation}
\Prod_{i=1}^{\frac{a-1}2}
\Prod_{j=1}^{b-1}
Q_c(\cos (\frac jb+\frac ia)\pi, \cos(\frac jb-\frac ia)\pi, \phi).
\end{equation}
$R_{a,b,c} \in \ZZ[\phi]$ and
$\cC(a,b,c,\phi)$ is singular iff $R_{a,b,c}(\phi)=0$.
\end{proposition}
\begin{proof}
$\phi \in \cZ_{a,b,c}$ iff there exists $(s,t)$ such that
$P_a(s,t)=P_b(s,t)=0$ and $Q_c(s,t,\phi)=0$. This conditions are equivalent to
have $t=\cos (\frac{j\pi}b+\frac{ i\pi}a)$ and $s=\cos (\frac{j\pi}b-\frac{ i\pi}a)$ and
$Q_c(s,t,\phi)=0$, for some
$1\leq i \leq \frac{a-1}2$ and $1\leq j \leq b-1$, from Proposition \ref{dp}.
\pn
$Q_c(s,t,\phi)$ is a symmetric polynomial of $\ZZ[\phi][t,s]$.
Let $\alpha_i = \frac{i\pi}a$, $\beta_j = \frac{j\pi}b$ and
$s=\cos(\alpha_i+\beta_j)$, $t=\cos(\alpha_i-\beta_j)$.
From $s+t = 2 \cos\alpha_i \cos\beta_j$ and
$st= \cos^2\alpha_i +\cos^2\beta_j -1$, we deduce that
$Q_c(s,t,\phi)$  belongs to $\ZZ[\phi,\cos\alpha_i] [\cos\beta_j]$.
$$R_i = \Prod_{j=1}^{b-1} Q_c(\cos(\alpha_i+\beta_j),\cos(\alpha_i-\beta_j),\phi)$$
belongs to $\ZZ[\phi,\cos\alpha_i]$ because the $\cos\beta_j$ are the roots of $V_b \in \ZZ[t]$.
From $Q_c(-s,-t,\phi)=Q_c(s,t,-\phi)$ we deduce that
$\Prod_{i=1}^{\frac{a-1}2} R_i(-\phi)R_i(\phi) = \Prod_{i=1}^{{a-1}}R_i(\phi) \in \ZZ[\phi]$.
We thus have $R^2_{a,b,c} \in \ZZ[\phi]$ and so it is for $R_{a,b,c}$.
\end{proof}
Let $s=\cos(\alpha+\beta)$, $t=\cos(\alpha-\beta)$.
Using Theorem \ref{fact} and Formula (\ref{fell}), we get
$$Q_c(s,t,\phi) = 2^{c-1} \Prod_{k=1}^{\pent{c}{2}} E_{\frac{2k\pi}n}(s,t).$$
Let us consider
$P_{\alpha,\beta,\gamma} = \Frac{1}{\sin^2\gamma} E_{2\gamma}(s +\phi,t+\phi).$
For $\gamma \not = \frac{\pi}2$, $P_{\alpha,\beta,\gamma}$ is
$$
\phi^2+2 \phi \cos\alpha \cos\beta+
\Frac{(\cos^2\alpha-\cos^2\gamma)(\cos^2\beta-\cos^2\gamma)}{\sin^2\gamma}$$
and
$$P_{\alpha,\beta,\frac\pi{2}}= \phi+\cos\alpha\cos\beta.$$
We therefore obtain
$$
Q_c(\cos (\alpha+\beta),\cos(\alpha-\beta),\phi)=
K \Prod_{k=1}^{\pent c2}
P_{\alpha,\beta,\frac{k\pi}{c}}(\phi)
$$
with $K = 2^{c-1}  \Prod_{k=1}^{c} 2\sin\frac{k\pi}c = c2^{c-1}$.
We get therefore
$$
R_{a,b,c}(\phi) = K^{\frac 12 (a-1)(b-1)}\Prod_{k=1}^{\pent c2} \Prod_{i=1}^{\frac{a-1}2}\Prod_{j=1}^{b-1} P_{\frac{i\pi}{a},\frac{j\pi}{b},\frac{k\pi}{c}}(\phi).
$$
We have written $R_{a,b,c}$ as the product of second or first-degree polynomials $P_{\alpha,\beta,\gamma}$ in
$\QQ(\cos\frac\pi{a},\cos\frac\pi{b},\cos\frac\pi{c})$.
\section{Computing the critical values}\label{ccv}
%
%
Our strategy consists in first computing separately the real roots of each $P_{\alpha,\beta,\gamma}$ and then combining these roots to get those of $R_{a,b,c}$.  A straightforward approach would be to use interval arithmetic to approximate the various trigonometric expressions, but this would fail when $R_{a,b,c}$ has multiple roots, unless we cannot ensure if
some discriminant or some resultant are null.
\subsection{Real roots of $P_{\alpha,\beta,\gamma}$}
Let $\alpha = \frac{i\pi}{a}$,
$\beta = \frac{j\pi}{b}$ and $\gamma =  \frac{k\pi}{c}$ with
$1 \leq i \leq \frac{a-1}2, \ 1 \leq j \leq b-1, \ 1 \leq k \leq \pent{c-1}2$.
\pn
If $\gamma=\frac\pi 2$, the unique root of $P_{\alpha,\beta,\frac{\pi}2}$ is $-\cos\alpha\cos\beta$.
If $\gamma \not = \frac{\pi}2$,
the discriminant of $P_{\alpha,\beta,\gamma}$ is
$4 \cos^2 \gamma \Bigl(1 -\Frac{\sin^2\alpha \sin^2 \beta }{\sin^2 \gamma}\Bigr)$.
It has the same sign as
\begin{equation}
\sin^2 \gamma -\sin^2\alpha \sin^2 \beta
\label{signdiscr}
\end{equation}
The knowledge of the sign of (\ref{signdiscr}) then gives explicit formulas for the
real roots of $P_{\alpha,\beta,\gamma}$.
\subsection{Multiplicity of 0}\label{multiple}
\begin{proposition}
The multiplicity of $\phi=0$ in $R_{a,b,c}$ is
$$
\frac{a-1}2((b,c)-1) + \pent{b}2((a,c)-1).
$$
\end{proposition}
\begin{proof}
We have to examine whenever $\phi=0$ is a root of $P_{\alpha, \beta, \gamma}$ where
$\alpha = \frac{i\pi}a$, $\beta=\frac{j\pi}b$ and $\gamma=\frac{k\pi}c$.
Here $a$ is odd so $\cos\alpha\not =0$. Thus, $\phi=0$ is a root of $P_{\alpha,\beta,\frac{\pi}{2}}$ if and only if
\begin{equation}
\cos\alpha\cos\beta \label{signPhi01}
\end{equation}
is null and when $\gamma\neq\frac{\pi}{2}$, $\phi=0$ is a root of $P_{\alpha,\beta,\gamma}$  if and only if the following expression is null
\begin{equation}
(\cos^2\alpha-\cos^2\gamma)(\cos^2\beta-\cos^2\gamma). \label{signPhi02}
\end{equation}
\begin{itemize}
\item If $\gamma=\beta=\frac\pi{2}$, $\phi=0$ is a root for $i=1,\ldots, \frac{a-1}2$.
\item If $\gamma \not = \frac{\pi}2$, $\phi=0$ is a root of $P_{\alpha,\beta,\gamma}$ if and only
if $\sin^2\gamma = \sin^2\alpha$ or $\sin^2\gamma = \sin^2\beta$, that is $ic=ka$ or $jc=kb$ or
$(b-j)c=kb$. The root $\phi=0$ is obtained for $i=\lambda \frac{a}{(a,c)}$, $k=\lambda \frac{c}{(a,c)}$,
$\lambda=1, \ldots, \frac{(a,c)-1}2$ and it is double when $\beta=\frac{\pi}2$.
It is also obtained for $j=\mu \frac{b}{(b,c)}$, $k=\mu\frac{c}{(a,c)}$, $\mu = 1, \ldots, (b,c)-1$. We obtain
$\pent b2 ((a,c)-1) + ((b,c)-1)(a-1)/2$.
\end{itemize}
We thus obtain the result.
\end{proof}
{\bf Remark.}
We find that $0$ is not a critical value if and only if $a$, $b$ and $c$ are pairwise coprime integers.
This result was first proved by Comstock (\cite{Com}, 1897), who found the number of crossing points of the
harmonic  curve parametrized by $x=T_a(t), y=T_b(t), z=T_c(t).$
\subsection{Non null multiple roots of $R_{a,b,c}$}
It may happen that $R_{a,b,c}$ has multiple root $\phi$. Several cases may occur.
\pn
$\blacktriangleright$
$P_{\alpha,\beta,\gamma}$ has a double root if and only if
$\Disc(P_{\alpha,\beta,\gamma})=0$, that is to say 
$\sin^2\gamma = \sin^2\alpha \sin^2\beta$. The double root is $\phi=-\cos\alpha\cos\beta$.
\pn
$\blacktriangleright$
$P_{\alpha,\beta,\gamma_1}$ and $P_{\alpha,\beta,\gamma_2}$ have a common root. In this case
$P_{\alpha,\beta,\gamma_1}=P_{\alpha,\beta,\gamma_2}$, that is to say
\begin{equation}
(\sin^2 \gamma_1 - \sin^2 \gamma_2)(\sin^2 \gamma_1 \sin^2 \gamma_2-\sin^2\alpha \sin^2 \beta)
\label{signresultant}
\end{equation}
is null.
\pn
$\blacktriangleright$
$P_{\alpha_1,\beta_1,\gamma_1}$ and $P_{\alpha_2,\beta_2,\gamma_2}$ have a common root.
\pn
The first two cases are related to the equation
\begin{equation}
\sin r_1 \pi \sin r_2 \pi = \sin r_3 \pi \sin r_4 \pi \label{sin4}
\end{equation}
where $r_i \in \QQ$. All the solutions of Equation (\ref{sin4})
are known (see \cite{My}). There is a one-parameter infinite family of solutions
corresponding to
$$
\sin \frac{\pi}6 \sin \theta = \sin\frac{\theta}2 \sin (\frac{\pi}2 - \frac{\theta}2),
$$
and a finite number of solutions listed in \cite{My}.
We deduce from a careful study of the Equation (\ref{sin4}):
\begin{proposition}\label{double1}
Let $\alpha = \frac{i\pi}a$, $\beta=\frac{j\pi}b$ and $\gamma=\frac{k\pi}c$, where $(a,b)=1$ and $a$ is odd.
$P_{\alpha,\beta,\gamma}$ has a double root iff $\beta = \frac{\pi}2$ and $\sin \gamma = \sin \alpha$.
In this case, the double root is $\phi=0$.
\end{proposition}
and
\begin{proposition}\label{double2}
Let $\alpha = \frac{i\pi}a$, $\beta=\frac{j\pi}b$ and $\gamma_1=\frac{k_1\pi}c$,
$\gamma_2=\frac{k_2\pi}c$, where $(a,b)=1$ and $a$ is odd. Then
$P_{\alpha,\beta,\gamma_1}$ and $P_{\alpha,\beta,\gamma_2}$ have a common root $\phi$ iff there are equal and
\bn
\item $\sin \alpha=\sin \gamma_1$, $\sin \beta = \sin \gamma_2$.\\ In this case, the roots are $\phi=0$ and $\phi = -2\cos\alpha \cos\beta$.
\item $\sin\beta = \frac 12$, $\sin \gamma_1 = \sin \frac 12 \alpha$, $\sin \gamma_2 = \cos \frac 12 \alpha$.\\
In this case the common roots are $\phi = -\cos(\alpha \pm \frac{\pi}6)$.
\item $\sin \gamma_2 = \sin \gamma_1$.
\en
\end{proposition}
$\blacktriangleright$
In case when $\alpha_1 \not = \alpha_2$ or $\beta_1 \not = \beta_2$, $P_{\alpha_1,\beta_1,\gamma_1}$ and $P_{\alpha_2,\beta_2,\gamma_2}$
have a common root if $\Res_{\phi} (P_{\alpha_1,\beta_1,\gamma_1},P_{\alpha_2,\beta_2,\gamma_2}) = 0$.
This resultant can be expanded and its sign is the one of:
\begin{equation}
\begin{array}{l}
\left (
(\cos^2\alpha_1-\cos^2\gamma_1)(\cos^2\beta_1-\cos^2\gamma_1)\sin^2\gamma_2 \right . - \\
\left . \quad\quad (\cos^2\alpha_2-\cos^2\gamma_2)(\cos^2\beta_2-\cos^2\gamma_2)\sin^2\gamma_1\right )^2\\
\quad -4(\cos\alpha_1\cos\beta_1 - \cos \alpha_2 \cos \beta_2)\sin^2\gamma_1\sin^2\gamma_2 \times \\
\quad \left ((\cos^2\alpha_1-\cos^2\gamma_1)(\cos^2\beta_1-\cos^2\gamma_1)\cos \alpha_2 \cos \beta_2 \sin^2\gamma_2\right . -\\
\quad\quad\left . (\cos^2\alpha_2-\cos^2\gamma_2)(\cos^2\beta_2-\cos^2\gamma_2)\cos\alpha_1\cos\beta_1 \sin^2\gamma_1\right ).
\end{array}\label{signresultant2}
\end{equation}
It would be interesting to get an arithmetic condition asserting that this resultant is null.
\subsection{Computing the diagrams}\label{cdiagrams}
Let $\phi \in \RR$. $\phi$ may be a rational number $r \in \QQ-\cZ_{a,b,c}$ or an algebraic number given
by a polynomial whom it is a root and an isolating interval.
The main step is the computation of the crossing nature at the double point $A_{\alpha,\beta}$ corresponding to parameters
$(t= \cos\alpha+\beta, s=\cos\alpha-\beta)$, where
$\alpha = \frac{i\pi}a$, $\beta=\frac{j\pi}b$.
There are two cases to consider.
\begin{enumerate}
\item We know the roots $\phi_1\leq \cdots \leq \phi_m$ of $Q_c(s,t,\phi)$.\\
If $\phi<\phi_1$ then $n=0$ otherwise let $n=\max\{k, \, \phi > \phi_k\}$. We have
$\sign {Q_c(s,t,\phi)} = (-1)^n$.
\item We do not know the roots of $Q_c(s,t,\phi)$. \\
We compute $Q_c(s,t,\phi)$ using the recurrence formula:
$$
\begin{array}{rcl}
Q_0 &=& 0, \, Q_1 = 1, \,
Q_2 = 2S+ 4 \phi, \\
Q_3 &=& -4\,T+12\,\phi\,S+4\,{S}^{2}+12\,{\phi}^{2}-3.\nonumber\\
Q_{n+4} &=& 2 \left(S + 2\,\phi \right)\left (Q_{n+3}+Q_{n+1} \right )\\
&&- 2 \left( 2\,{\phi}^{2}+ 2\,T+2\,\phi\,S + 1\right) Q_{n+2} - Q_n.\label{Q_n}
\end{array}
$$
where $S=s+t=2\cos\alpha\cos\beta$ and $T=st=\cos^2\alpha + \cos^2\beta -1$  (see \cite{KPR}).
We work formally in $\QQ[u,v]/\langle M,N\rangle$ where
$M, N$ are the minimal polynomials of $u=\cos\alpha$, $v=\cos\beta$.
\end{enumerate}
The sign of the crossing is
\begin{eqnarray*}
D(s,t,\phi) &=& Q_c(s,t,\phi) P_{b-a}(s,t,\phi)\\
&=& (-1)^{i+j} \sin\frac{ib\pi}a \sin\frac{ja\pi}b Q_c(s,t, \phi) \\
&=& (-1)^{i+j + \pent{ib}a + \pent{ja}b} Q_c(s,t, \phi).
\end{eqnarray*}
\section{The algorithm}\label{algo}
We want to compute all the real roots $\phi_1< \ldots <\phi_n$ of $R_{a,b,c}$ that factors in $\frac 12 (a-1)(b-1)\pent c2$ polynomials
$P_{\alpha_i, \beta_j, \gamma_k}$.
We precisely want non overlapping intervals $[a_m,b_m]$ for these roots in order to chose sample rational points
$r_0 < a_1$, $b_i<r_i<a_{i+1}$, $b_n<r_n$.

At some stages, one may need to compute the sign of $\Disc(P_{\alpha,\beta,\gamma})$ (expression (\ref{signdiscr}))
or $\Res(P_{\alpha_1,\beta_1,\gamma_1},P_{\alpha_1,\beta_1,\gamma_1})$ (expressions (\ref{signresultant}) and (\ref{signresultant2})) in order to decide whether two roots are distinct or not.
This information is required for two reasons.
We first want to be sure that we get all the roots and secondly, we will need to know all the roots of
$Q_c(\cos(\alpha_i+\beta_j),\cos(\alpha_i-\beta_j),\phi)$ with their multiplicities in order to determine the nature of
the crossing over the corresponding double point in the diagram (section \ref{cdiagrams}.

The signs of  (\ref{signPhi01}) and (\ref{signPhi02}), may be evaluated by simple arithmetic considerations
on $\alpha$, $\beta$, $\gamma$.
\pn
{\bf {{\em Isolate} and {\em Refine}}}.
A very first step is to get accurate isolating intervals  with rational bounds
for $\cos \alpha_i$, $\cos \beta_j$ and $\cos \gamma_k$ to perform interval arithmetic
for the real roots of $P_{\alpha_i, \beta_j, \gamma_k}$.

Such intervals can be computed by performing algorithms based on Descarte rule of signs (see for example \cite{RZ}) on the used Chebyshev polynomials $V_n$. Algorithms like in \cite{RZ} can easily solve such polynomials for very high degrees (several thousands) with a large accuracy.
The computation of the required isolating intervals can then be performed as a pre-processing for the global algorithm.

From now, we denote by {\tt Isolate}($P$,$acc$) the function that isolates the roots of a univariate polynomial $P$ with rational coefficients by means of intervals with rational bounds for a given accuracy $acc$ (maximal length of the intervals). This function provides non overlapping intervals that contains a unique real root of $P$ (and such that each real root of $P$ is contained in one of the intervals).

Note that if more accuracy is required for some intervals,
it is easy to refine them  from the isolating intervals provided by the function {\tt Isolate}:
just evaluating $P$ at some points, without running again the {\tt Isolate} function with a higher value for $acc$).
We name by {\tt Refine}($I$,$P$,$acc$) the function that decreases the length of the interval $I$ to get an accuracy  $\leq acc$, knowing that the interval isolates a real root of $P$.
\pn
{\bf {\em IsolateP}}.
Thanks to Proposition \ref{double1}, one can compute the roots of $P_{\alpha,\beta,\gamma}$ with an appropriate accuracy,
using multi-precision interval arithmetic for the evaluations. We will use the function
{\tt IsolateP}$(\alpha,\beta,\gamma,acc)$
that returns a (possibly empty) list of $(\alpha,\beta,\gamma,[u,v])$ corresponding to isolating intervals $[u,v]$ for the roots.
\pn
{\bf {\em SignTest}}.
When two isolating intervals $[u_1,v_1]$ and $[u_2,v_2]$ corresponding to  $(\alpha_1,\beta_1,\gamma_1)$ and
$(\alpha_2,\beta_2,\gamma_2)$ are such that
$[u_1,v_1] \bigcap [u_2,v_2] \not = \emptyset$,
we first use a filter (named {\tt SignTest} in the sequel) which consists in using multi-precision interval arithmetic for
the evaluation of $\Res_{\phi} (P_{\alpha_1,\beta_1,\gamma_1},P_{\alpha_2,\beta_2,\gamma_2})$ (expressions
\ref{signresultant}) and (\ref{signresultant2}).

Thanks to Proposition \ref{double2}, we know by arithmetic considerations when (\ref{signresultant}) is null. We know also
the corresponding common roots and we change $[u_i,v_i]$ to $[u_1,v_1] \bigcap [u_2,v_2]$.

Expression (\ref{signresultant2}) is \\
$P = \left (
(C^2_1-C^2_5)(C^2_3-C^2_5)(1-C^2_6) \right . - $\\
\hbox{} \hfill $\quad\left . (C^2_2-C^2_6)(C^2_4-C^2_6)(1-C^2_5)\right )^2$\quad \\
\hbox{}\hfill$  - 4 \, (C_1C_3 - C_2 C_4)(1-C^2_5)(1-C^2_6) \times $ \hfill\hbox{}\\
$\left ((C^2_1-C^2_5)(C^2_3-C^2_5)C_2 C_4 (1-C^2_6)\right . -$\\
\hbox{}\hfill $\left . (C^2_2-C^2_6)(C^2_4-C^2_6)C_1 C_3 (1-C^2_5)\right ),\quad$\\
where $C_1 = \cos\alpha_1$,
$C_2 = \cos\alpha_2$, $C_3 = \cos\beta_1$, $C_4 = \cos\beta_2$, $C_5 = \cos\gamma_1$, $C_6 = \cos\gamma_2$.

Given isolating intervals with rational bounds that contain the
values of the required $C_i, \, i=1,\ldots,6$, the function {\tt SignTest}$(\alpha_1,\alpha_2,\beta_1,\beta_2,\gamma_1,\gamma_2)$
straightforwardly evaluates $P$. 
If the resulting interval is $[0, 0]$ or do not contains $0$, one can decide the sign of the input, otherwise, the function
returns {\tt FAIL}.
\pn
{\bf{\em FormalNullTest}}. In case of failure of {\tt SignTest}, one has to decide if the input
is null or not, which is the goal of the function {\tt FormalNullTest}  we now describe.

Let us write $\alpha_1 =  \frac{i_1\pi}{a_1},\alpha_2= \frac{i_2 \pi}{a_2}, \beta_1 = \frac{j_1 \pi}{b_1}, \beta_2=\frac{j_2 \pi}{b_2}, \gamma_1 =\frac{k_1\pi}{c_1}, \gamma_2 = \frac{k_2 \pi}{c_2}$.
Let $m$ be the smallest common multiple of $a_1, a_2, b_1, b_2, c_1$ and $c_2$. According to the
definitions of $T_n$, we have $C_i = T_{n_i}(\cos \frac{\pi}{m})$.
Since $M_m$ is the minimal polynomial of $\cos\frac{\pi}{m}$,
the expression $P(C_i,\ldots,C_6)$ is null if and only if
$P (T_{n_1},\ldots,T_{n_6})= 0$ in $\QQ[t]/\langle M_m (t)\rangle$.
\pn
{\bf{\em DoubleTest}}.
Our function first performs the {\tt SignTest}. If it returns an interval with bounds of same sign, then the sign of the tested expression is the sign of the two bounds of the interval.
Otherwise, we run the {\tt FormalNullTest}. If this test returns $0$ then the expression is null.
Otherwise, we decrease the lengths of the intervals that represent the values of $\cos\frac{k\pi}{m}$
by calling the function {\tt Refine} until the {\tt SignTest} does not FAIL
(the fact that the sign of the expression to be tested is known not to be $0$ ensures that this process will end).
\pn
{\bf The global algorithm}. We proceed in three steps :

(0) We isolate the roots of some Chebyshev polynomials using the {\tt Isolate} black-box with an arbitrary accuracy.

(1) We compute separately the roots of the $P_{\alpha,\beta,\gamma}$ by using {\tt IsolateP}.

(2) We then consider the list of these roots and observe carefully the overlapping intervals.
For any pair of overlapping interval, we decide whether corresponding resultants are null or not using
{\tt DoubleTest}. If the corresponding roots are equal then we change their isolating intervals by taking their
intersection.
\pn
From these disjoint intervals with rational bounds, we straightforwardly get the roots with their multiplicities.
We thus deduce the sample points $r_0 , \ldots, r_n$ we need. Furthermore, for each
$\alpha_i = \frac{i\pi}a$, $\beta_j = \frac{j\pi}b$, we know the roots with their multiplicities of
$Q_c(t,s,\phi)$, where $t=\cos(\alpha_i + \beta_j)$ and $s=\cos(\alpha_i - \beta_j)$. This information is helpful
for knowing the crossing nature at the point $A_{\alpha_i,\beta_j}$ (section \ref{cdiagrams}).
\section{Experiments}\label{experiments}
In the appendix of \cite{KPR}, we gave parametrizations of every rational knot as $\cC(3,b,c,\phi)$ and
$\cC(4,b,c,\phi)$ where $(b,c)$ were minimal for the lexicographic order ($c\leq 300$).
For 6 knots we knew the minimal $b$ and that $c>300$.
With the method we developed here, we recover all the minimal parametrizations we gave in \cite{KPR} but also for the
6 missing knots. The following knots admit the parametrizations:
$$
\begin{array}{ll}
9_{5}=\cC(3,13, 326, 1/85),&
10_3 = \cC(4,13, 348, 1/138),\\
10_{30} = \cC(4,13, 306, 1/738),&
10_{33} = \cC(4,13,856,1/328),\\
10_{36}=\cC(3,14, 385, 1/146),&
10_{39}=\cC(3,14, 373, 1/182).
\end{array}
$$
For example, one deduces that there is no parametrization of $9_5$ as Chebyshev knots with $(a,b,c) <_{\mathrm{lex}} (3,13,326)$.
\pn
$R_{3,14,385}$ has degree 4992. It has 2883 real roots. All are simple roots except 0 that is of multiplicity
$6$.
\pn
$R_{4,13,856}$ has degree 15390 and 9246 real roots ($0$ has multiplicity 18). We get 2050 non trivial knots, 83 of them are distinct, and 63 have less than 10 crossings. The total running time ---
critical values with their multiplicities, sampling of 1442 values, computing knot invariant --- was 450"
({\sc Maple 13}, on Laptop, 3Gb of RAM, 3MHz).
\pn
Outside the intrinsic combinatorial aspects of the problem, the complexity of our algorithm essentially depends on the {\tt FormalNullTest}.
In the worst case $d = abc$ and $\deg M_d = \frac 12(a-1)(b-1)(c-1)$, when $a,b, c$ are prime integers,
the most difficult computation consists in deciding if the expression \ref{signresultant2} is null
or not which is equivalent to testing if a univariate polynomial of degree at most $4d$ is null modulo $M_d$ or not.

These computations can be speed up a lot since they can be performed modulo a prime integer: all the considered polynomials have a power of two as leading coefficient and we just need to test if one polynomial is null modulo another one.

In this challenging experiments, we never had to run the {\tt FormalNullTest}, the {\tt SignTest} being always sufficient, thanks to the filters given by propositions \ref{double1} and \ref{double2} and to a good (experimental) choice initial choice of accuracy when computing the prerequisites running the {\tt Isolate} algorithm.
\section{Conclusion}\label{conc}
The method we developed in this paper allows us to compute Chebyshev knot diagrams for high values of $a$, $b$ and $c$.
Our experience with small $a$ and $b$ shows that the difficult cases (multiple roots of $R_{a,b,c}$ we found) were predictable.
There are certainly some specific reasons connected with arithmetic properties and the structure of cyclic extensions.
\pn
The main difference with the algorithm described in \cite{KPR} and the computation of $R_{a,b,c}$ as a polynomial
of degree $\frac 12(a-1)(b-1)(c-1)$, is that it came as a resultant of a polynomial
of degree $(c-1)$ in $(X,\phi)$ and a polynomial of degree $\frac 12 (a-1)(b-1)$ in $X$ with coefficients in a unique field extension.
The example described in this section can be considered as the extremal case, in terms of degree, to be solved using methods
from the state of the art when running \cite{KPR} while it can be solved in few minutes with the method proposed in this article.
\pn
From the point of view of knot theory, it is proved in \cite{KP4}, that rational knots with $N$ crossings can be parametrized by polynomials of degrees $(3,b,c)$ where $b+c \leq 3N$ which is far better that the results we obtain here. But we challenged to give, as it was done with Lissajous knots in \cite{BDHZ}, an exhaustive and certified list of minimal parametrizations.
We consider that it might be one step in the computing of polynomial curves topology.

%
%
\end{document}